\documentclass[11pt]{amsart}
\usepackage{amsmath}
\usepackage{amssymb, verbatim}
\usepackage{color}
\usepackage{graphicx}

\title[]{Concavity of the Lagrangian Phase Operator and Applications}
\author{Tristan C. Collins}
\address{Department of Mathematics, Harvard University, 1 Oxford St., Cambridge, MA}
\thanks{Supported in part by National Science Foundation grants DMS-1506652 (T.C.C.) and DMS-12-66033 (S.P.).}
\email{tcollins@math.harvard.edu}
\author{Sebastien Picard}
\address{Department of Mathematics, Columbia University, 2990 Broadway, New York, NY}
\email{picard@math.columbia.edu}
\author{Xuan Wu}
\email{xuanw@math.columbia.edu}

\theoremstyle{plain}
\newtheorem{thm}{Theorem}[section]
\newtheorem{prop}[thm]{Proposition}

\newtheorem{lem}[thm]{Lemma}
\newtheorem{cor}[thm]{Corollary}

\theoremstyle{definition}

\numberwithin{equation}{section}

\newcommand{\del}{\partial}
\newcommand{\p}{\partial}

\newcommand{\la}{\lambda}
\newcommand{\dbar}{\overline{\del}}

\renewcommand{\leq}{\leqslant}
\renewcommand{\geq}{\geqslant}

\renewcommand{\epsilon}{\varepsilon}
\renewcommand{\phi}{\varphi}

\begin{document}

\begin{abstract}
We study the Dirichlet problem for the Lagrangian phase operator, in both the real and complex setting.  Our main result states that if $\Omega$ is a compact domain in $\mathbb{R}^{n}$ or $\mathbb{C}^n$, then there exists a solution to the Dirichlet problem with right-hand side $h(x)$ satisfying $|h(x)| > (n-2)\frac{\pi}{2}$ and boundary data $\phi$ if and only if there exists a subsolution.
\end{abstract}

\maketitle

\section{Introduction}
In this paper we study the Dirichlet problem, in both the real and complex settings, for a certain non-linear elliptic operator which we call the Lagrangian phase operator.  The Dirichlet problem for a broad class of fully nonlinear equations was studied in the groundbreaking paper of Caffarelli-Nirenberg-Spruck ~\cite{CNS3}. They considered equations  of the form
\begin{equation}\label{eq:GenDirProb}
F(D^2 u) = h(x),\ \ u|_{\p \Omega} = \phi
\end{equation}
for an unknown function $u:\Omega \rightarrow \mathbb{R}$ under various conditions on $F$ and $\Omega$, generalizing previous work of ~\cite{CNS1,Ivo1, Ivo2, K2} and others on Monge-Amp\`ere equations. Since then, the Dirichlet problem for elliptic operators with various structural constraints has been studied by many authors, including generalizations to a larger class of equations and domains ~\cite{CW,G2,G4,ITW,T1}.  In each of these works, the requirement of concavity of $F$ on the space of symmetric matrices is one of the essential requirements for the solvability of the equation.
\par 
In this paper, we consider the Dirichlet problem for the Lagrangian phase operator in both the real and complex setting. Namely, in the real case, let $\Omega \subset \mathbb{R}^{n}$ be a compact domain, and suppose $u:\Omega \rightarrow \mathbb{R}$.  Let $\lambda_1, \dots, \lambda_n$ denote the eigenvalues of the Hessian $D^2 u$. We consider the boundary value problem
\begin{equation}\label{eq:DirProb}
\begin{aligned}
F(D^{2}u) := \sum_{i=1}^{n} \arctan \lambda_i &= h(x),\\
u|_{\del \Omega} &= \phi.
\end{aligned}
\end{equation}
In the complex case, we consider the same equation where now $\Omega \subset \mathbb{C}^{n}$, and $\lambda_1, \ldots, \lambda_n$ are the eigenvalues of the complex Hessian $\del\dbar u$.  In both cases, the properties of the operator $F(\cdot)$ are intimately tied to the range of the right-hand side of~\eqref{eq:DirProb}.  To be precise, if $h(x) \geq (n-1)\frac{\pi}{2}$, then $F$ is concave, while if $(n-2)\frac{\pi}{2} \leq h(x) \leq (n-1)\frac{\pi}{2}$, then $F$ will have concave level sets \cite{Y1}, but can fail to be concave in general (see Lemma~\ref{lem:cone_facts} below).  Furthermore, if $0\leq h(x) < (n-2)\frac{\pi}{2}$, then $F$ fails to have even concave level sets, and examples of Nadirashvili-Vl\u{a}du\c{t} \cite{NV} and Wang-Yuan \cite{WY1} show that solutions of \eqref{eq:DirProb} can fail to have interior estimates.  Thus, from the analytic point of view, it is natural to restrict our study of~\eqref{eq:DirProb} to the case when $h(x) > (n-2)\frac{\pi}{2}$; following \cite{JY, CJY, Y1} we call this the ``supercritical phase" condition.  Note that $h(x) < -(n-2)\frac{\pi}{2}$ can be treated similarly.
\par
The work of Caffarelli-Nirenberg-Spruck \cite{CNS3} shows that there is an intimate connection between the geometry of the domain $\Omega$, and the solvability of the Dirichlet problem of a general concave, elliptic operator.  Consider a motivating, simple example: in order to solve~\eqref{eq:GenDirProb} for the real Monge-Amp\`ere operator $F(D^2u) = \log(\det D^{2}u)$, there must exist a convex function $\underline{u} : \Omega \rightarrow \mathbb{R}$ with $\underline{u}|_{\del \Omega} = \phi$.  In particular, if $\phi \equiv 0$, then $\Omega$ must be a convex domain in $\mathbb{R}^{n}$.  Subsequently, Guan \cite{G4} showed that the conditions imposed in \cite{CNS3} on the geometry of $\Omega$ can be dropped, provided one assumes instead the existence of an admissible subsolution.  This idea was subsequently extended to Riemannian manifolds with boundary \cite{G2,G3}, Monge-Amp\`ere type equations on complex manifolds \cite{GuanCMA,GL,GS}, and to a class of fully non-linear elliptic equations defined on domains in $\mathbb{C}^{n}$ \cite{LiSY}, to name just a few.  There are also extensions to compact Riemannian, and Hermitian manifolds where the notion of a subsolution needs to be amended \cite{G1, S1}.
\par
In this paper we apply these ideas to the Lagrangian phase operator.  Precisely, we prove
\begin{thm}\label{thm:DirProb}
Let $\Omega \subset \mathbb{R}^n$ be a $C^4$ bounded domain. Let $\phi: \del \Omega \rightarrow \mathbb{R}$ be in $C^4(\del \Omega)$ and $h: \overline{\Omega} \rightarrow [(n-2){\pi \over 2} + \delta, \, n \frac{\pi}{2})$ in $C^{2}(\overline{\Omega})$,  where $\delta \in \mathbb{R}_{>0}$. If there exists a function $\underline{u}: \Omega \rightarrow \mathbb{R} $ in $C^4(\overline{\Omega})$ such that
\begin{equation}\label{sub_def}
\begin{aligned}
\sum_i \arctan \underline{\lambda}_i &\geq h(x) \ {\rm in} \ \Omega,\\
\underline{u}|_{\del \Omega} &= \phi \ \ {\rm on} \ \del \Omega,
\end{aligned}
\end{equation}
where $\underline{\lambda}_i$  are the eigenvalues of $D^2\underline{u}$, then the Dirichlet problem~\eqref{eq:DirProb} admits a unique $C^{3,\alpha}(\overline{\Omega})$ solution. If $\Omega$, $h$, $\phi$, and $\underline{u}$ are smooth, then the solution $u$ is smooth.
\end{thm}
\par
Similarly, in the complex case we prove
\par
\begin{thm}\label{thm:CDirProb}
Let $\Omega \subset \mathbb{C}^n$ be a $C^4$ bounded domain. Let $\phi: \del \Omega \rightarrow \mathbb{R}$ be in $C^4(\del \Omega)$ and $h: \overline{\Omega} \rightarrow [(n-2){\pi \over 2} + \delta, \, n \frac{\pi}{2})$ in $C^2(\overline{\Omega})$, where $\delta \in \mathbb{R}_{>0}$. If there exists a function $\underline{u}: \Omega \rightarrow \mathbb{R}$ in $C^4(\overline{\Omega})$ such that
\begin{equation}
\begin{aligned}
\sum_i \arctan \underline{\lambda}_i &\geq h(z) \ {\rm in} \ \Omega,\\
\underline{u}|_{\del \Omega} &= \phi \ \ {\rm on} \ \del \Omega,
\end{aligned}
\end{equation}
where $\underline{\lambda}_i$  are the eigenvalues of $\del\dbar\underline{u}$, then there exists a unique $C^{3,\alpha}(\overline{\Omega})$ solution $u: \Omega \rightarrow \mathbb{R}$ of the Dirichlet problem
\begin{equation}
\begin{aligned}
\sum_{i=1}^{n} \arctan \lambda_i &= h(z) \ {\rm in} \ \Omega,\\
\underline{u}|_{\del \Omega} &= \phi \ \ {\rm on} \ \del \Omega,
\end{aligned}
\end{equation}
where $\lambda_i$ are the eigenvalues of $\del\dbar u$. If $\Omega$, $h$, $\phi$, and $\underline{u}$ are smooth, then the solution $u$ is smooth.
\end{thm}
\par
Let us provide some geometric motivation for studying the Lagrangian phase operator.  In the real case, the special Lagrangian equation 
\[
\sum_i \arctan \lambda_i = \Theta,
\]
was introduced by Harvey-Lawson ~\cite{HL} in the study of calibrated geometries. Here $\Theta$ is a constant called the phase angle and $\lambda_i$ are, as before, the eigenvalues of $D^{2}u$.  In this case the graph $x \mapsto (x, \nabla u(x))$ defines a calibrated, minimal submanifold of $\mathbb{R}^{2n}$. Since the work of Harvey-Lawson, special Lagrangian manifolds have gained wide interest, due in large part to their fundamental role in the Strominger-Yau-Zaslow description of mirror symmetry \cite{SYZ}.  It is well known that any special Lagrangian manifold can locally be represented as the graph of a potential function $u$ solving the equation $F(D^2u) = \Theta$, where $F$ is the operator appearing in~\eqref{eq:DirProb}-- however, the potential function $u$ depends on a choice of a Lagrangian subspace.  This means that, in practice, one can obtain smooth special Lagrangian submanifolds whose potential function looks singular when written with respect to a certain choice of subspace.  For this reason, the problem of finding special Lagrangian submanifolds seems to require a more geometric approach.  One such approach is via the the Lagrangian mean curvature flow; we refer the reader to \cite{MTWsurv} and the references therein for an introduction to the vast literature in this active area of research.  In the case of the Dirichlet problem, these subtleties do not arise.
\par
There are several works by various authors which are related to Theorem ~\ref{thm:DirProb}. The Dirichlet problem~\eqref{eq:DirProb} was solved by Caffarelli-Nirenberg-Spruck \cite{CNS3} for $h(x) = (n-2)\frac{\pi}{2}$ when $n$ is even, and $h(x) = (n-1)\frac{\pi}{2}$ when $n$ is odd, under a condition on the geometry of the domain $\Omega$. In \cite{GZ}, Guan-Zhang solve the Dirichlet problem \eqref{eq:DirProb} with supercritical $h(x)$ for domains $\Omega \subset \mathbb{R}^3$ by studying another equivalent equation. Interior $C^1$ estimates for the minimal surface system were first established by M.-T. Wang in \cite{MTW}.  Interior estimates for the special Lagrangian equation with supercritical phase have been obtained by Warren-Yuan ~\cite{WY2} for the $C^1$ estimate and Wang-Yuan ~\cite{WY} for the $C^2$ estimate. In ~\cite{BW}, Brendle-Warren study a boundary value problem for the special Lagrangian equation in which the boundary data involves specifying the graph of the gradient map of the solution.
\par
In the complex setting, the Dirichlet problem solved in Theorem ~\ref{thm:CDirProb} is a local version of the deformed Hermitian-Yang-Mills (dHYM) equation for a holomorphic line bundle over a compact K\"ahler manifold.  The dHYM equation was discovered by Marino-Moore-Minasian-Strominger \cite{MMMS} as the requirement for a $D$-brane on the B-model of mirror symmetry to be supersymmetric (or BPS).  It was shown by Leung-Yau-Zaslow \cite{LYZ} that, in the semi-flat model of Mirror Symmetry, solutions of the dHYM equation correspond by a Fourier-Mukai transform to special Lagrangian submanifolds of the mirror. The study of the dHYM equation was initiated by ~\cite{JY} and recently solved, in the supercritical phase case, by Jacob, Yau, and the first author ~\cite{CJY} assuming a suitable notion of subsolution. The notion of subsolution from ~\cite{CJY} differs from ours since they study the equation on a closed manifold without boundary.  By contrast, the main difficulty in \cite{CJY} is the proof of the (interior) $C^2$ estimate, while here the main difficulty occurs in the boundary $C^2$ estimates, which have a rather different flavour.
\par
One of the hurdles arising in the current work is the lack of concavity of the operator $F$ appearing in~\eqref{eq:DirProb}. One of the key elements in the present paper is to transform the equation to find hidden concavity properties when $h(x) >  {(n-2)\pi \over 2}$; namely, we introduce an elliptic operator $G$ which is indeed concave (see Corollary ~\ref{cor:Gconcave}), and whose level sets agree with the level sets of $F$. Given concavity of $G$, our Theorem~\ref{thm:DirProb} is closely related to a general theorem of Guan \cite{G2}, though as remarked after Lemma~\ref{Concave_G}, the operator $G$ does not fit all of the general structural conditions imposed there.  In the complex case the operator $G$ is rather far from fitting the structural assumptions imposed in \cite{LiSY}.  As a result, 
we provide detailed proofs of both Theorem~\ref{thm:DirProb} and Theorem~\ref{thm:CDirProb}.
\par
In order to prove Theorem~\ref{thm:DirProb}, we derive a priori estimates for solutions of~\eqref{eq:DirProb}. The main difficulty is estimating the second derivatives of the solution at the boundary of the domain $\Omega$. The technique involved follows ideas of Guan ~\cite{G1,G2,G3} and Trudinger ~\cite{T1}. Once a priori estimates are obtained, Theorem~\ref{thm:DirProb} is obtained from a standard continuity method argument.
\par
In \S \ref{cplx_diri}, we prove Theorem~\ref{thm:CDirProb}. We adapt our argument from the real case and also use ideas of S.Y. Li \cite{LiSY}. We outline the steps which are the same as the real case, and carefully handle the new difficulties which arise in the complex setting.
\bigskip
\par
{\bf Acknowledgements:} We would like to thank D.H. Phong for all his guidance and support. We also thank Pei-Ken Hung for many helpful discussions.  The authors are grateful to Valentino Tosatti and Mu-Tao Wang for helpful comments.

\section{The Angle Operator and Hidden Concavity Properties}
We begin by stating some facts about functions with supercritical Lagrangian phase-- namely, functions satisfying
\[
F(D^2u) := \sum_{i} \arctan (\la_i) = h(x)
\]
with $h(x) >(n-2)\frac{\pi}{2}$. These properties are well-known and can be found in \cite{WY,Y1}.
\begin{lem}\label{lem:cone_facts}
Suppose $\lambda_1 \geq \lambda_2 \geq \cdots \geq \lambda_n$ are such that $\sum_i \arctan \lambda_i \geq (n-2) {\pi \over 2} + \delta$. The following properties hold (with constants depending on $\delta>0$)
\begin{enumerate}
\item $\lambda_1 \geq \lambda_2 \geq \cdots \geq \lambda_{n-1} >0, \ \ |\lambda_n| \leq \lambda_{n-1}$,
\item $\sum_i \lambda_i \geq 0$,
\item $\lambda_n \geq -C(\delta)$,
\item if $ \lambda_n < 0$, then $\sum_i \frac{1}{\lambda_i}\leq -\tan(\delta)$,
\item For any $\sigma \in ((n-2) {\pi \over 2}, n {\pi \over 2})$, we define the set 
\[ \Gamma^\sigma=\{\lambda \in \mathbb{R}^n: \sum_i \arctan \lambda_i>\sigma \}.
\]
Then $\Gamma^{\sigma}$ is a convex set, and $\del \Gamma^{\sigma}$ is a smooth convex hypersurface. 
\end{enumerate}
\end{lem}
\begin{proof} It is easy to see that the first $(n-1)$ eigenvalues must be positive. Next, we notice
\[
(n-2) {\pi \over 2} + \delta \leq \sum_i \arctan \lambda_i \leq (n-2) {\pi \over 2} + \arctan \lambda_{n-1} + \arctan \lambda_n.
\]
It follows that $\arctan \lambda_{n-1} + \arctan \lambda_n \geq 0$. This proves $|\lambda_{n}| \leq \lambda_{n-1}$. It follows that $\sum_i \lambda_i \geq 0$. From the above inequality, it also follows that $- {\pi \over 2} + \delta \leq \arctan \lambda_n$. This implies $\lambda_n \geq -C$. For 4, let $\theta_i=\arctan\lambda_i$, so that $\sum\theta_i\geq (n-2)\frac{\pi}{2} +\delta$. Estimate
\[
\begin{aligned}
\theta_n+\frac{\pi}{2} &\geq (n-2)\frac{\pi}{2}+\delta+\frac{\pi}{2}-\sum_{i\leq n-1}\theta_i\\
&= \delta+\sum_{i\leq n-1} \bigg( \frac{\pi}{2}-\theta_i \bigg).
\end{aligned}
\]
Assuming $-\frac{\pi}{2}<\theta_n<0$, we have
\[
0< \theta_n + {\pi \over 2} < {\pi \over 2}.
\]
Hence
\[
-\frac{1}{\lambda_n} = \tan \bigg( \theta_n + \frac{\pi}{2} \bigg) \geq \tan \bigg( \delta+\sum_{i\leq n-1}(\frac{\pi}{2}-\theta_i) \bigg) \geq \tan\delta + \sum_{i\leq n-1} \frac{1}{\lambda_i},
\]
and 4 follows. Property 5 is due to Yuan (Lemma 2.1 in \cite{Y1}).
\end{proof}
 The following lemma will be turn out to be key, as it allows us to transform $F$ into a concave operator.

\begin{lem}\label{Concave_G}
Let $f(\lambda) = \sum_i \arctan \lambda_i$ be defined on $\{\lambda \in \mathbb{R}^n : \sum_i \arctan \lambda_i \geq (n-2) {\pi \over 2} + \delta \}$. Then there exists $A$ large enough depending on $\delta$ such that $g(\lambda) =-e^{-A f(\lambda)}$ is a concave function.
\end{lem}
\begin{proof}
It suffices to show $\frac{\partial^2 g}{\partial\lambda_i\partial\lambda_j}$ is negative definite. We calculate as follows:
\[
\begin{aligned}
\frac{\partial g}{\partial \lambda_i}&=Ae^{-Af}\frac{1}{1+\lambda_i^2},\\
\frac{\partial^2 g}{\partial \lambda_i\partial\lambda_j}&=-Ae^{-Af}\bigg(\frac{A+2\lambda_i\delta_{ij}}{(1+\lambda_i^2)(1+\lambda_j^2)} \bigg).
\end{aligned}
\]
Define
\[
H_{ij} = \bigg(\frac{A+2\lambda_i\delta_{ij}}{(1+\lambda_i^2)(1+\lambda_j^2)} \bigg).
\]
We will show that all leading principal minors of $H_{ij}$ are positive. First, compute
\begin{equation} \label{det_pullout}
\det H_{ij} = \frac{1}{\Pi(1+\lambda_i^2)^2} \, \det(A+2\lambda_i\delta_{ij}).
\end{equation}
We claim that 
\begin{equation} \label{det_identity}
\det(A+2\lambda_i\delta_{ij})= A 2^{n-1}\sigma_{n-1}(\lambda) +2^n \sigma_n(\lambda),
\end{equation}
where $\sigma_k(\mu)$ is the $k$th elementary symmetric polynomial of $\mu \in \mathbb{R}^n$ if $k \leq n$, and zero if $k > n$. Identity (\ref{det_identity}) can be proved by induction. The case of a $1 \times 1$ matrix is trivial, so we assume the identity holds for a $(n-1) \times (n-1)$ matrix. We may write $\det(A+2\lambda_i\delta_{ij})$ as the sum
\[
\det \begin{pmatrix}
2\lambda_1 & A & \cdots&A\\
0 & A+2\lambda_2&\cdots&A\\
\vdots & \vdots& \ddots& \vdots\\
0 &A&\cdots& A+2\lambda_n
\end{pmatrix} +
 \det \begin{pmatrix}
A & A & \cdots&A\\
A & A+2\lambda_2&\cdots&A\\
\vdots & \vdots& \ddots& \vdots\\
A &A&\cdots& A+2\lambda_n
\end{pmatrix}.
\]
We use the induction hypothesis on the first determinant and row reduction on the second to obtain
\[
\det(A+2\lambda_i\delta_{ij}) = 2 \lambda_1 \{ 2^{n-2} A \sigma_{n-2}(\lambda|1) + 2^{n-1} \sigma_{n-1}(\lambda|1) \} + 2^{n-1}A \sigma_{n-1}(\lambda|1).
\]
Here we introduced the notation $\sigma_{k} (\lambda | i )$ to denote the $k$-th elementary function of $(\lambda | i) = (\lambda_1, \cdots, \widehat{\lambda_{i}}, \cdots, \lambda_n)\in \mathbb{R}^{n-1}$. By the identity
\[
\sigma_{k+1}(\lambda)= \lambda_1 \sigma_{k}(\lambda|1) + \sigma_{k+1}(\lambda|1),
\]
we obtain (\ref{det_identity}). 
\par
Without loss of generality we may assume that $ \la_1\geq \la_2 \geq \ldots \geq \la_n$. By Lemma~\ref{lem:cone_facts}, only the smallest eigenvalue $\lambda_n$ could be negative. From (\ref{det_pullout}) and (\ref{det_identity}), we see that all leading principal minors of $H_{ij}$ up to order $(n-1)$ are positive, and the full determinant is positive if $\lambda_n \geq 0$. If $\lambda_n < 0$, then $\sigma_n(\lambda) < 0$ and by (\ref{det_pullout}) and (\ref{det_identity}), we have
\[
\det H_{ij} = \frac{2^{n-1}\sigma_n(\lambda)}{\Pi(1+\lambda_i^2)^2}(A\sum\frac{1}{\lambda_i}+2). \]
By property 4 in Lemma~\ref{lem:cone_facts}, we know that $\sum\frac{1}{\lambda_i} \leq - \tan \delta$. By choosing $A$ large enough, we see that $H_{ij}$ is positive definite. Hence $g$ is a concave function.
\end{proof}

Note, that the constant $A$ cannot be chosen on any symmetric cone $\Gamma \subset \mathbb{R}^{n}$ containing the set $E:= \{\lambda \in \mathbb{R}^n : \sum_i \arctan \lambda_i \geq (n-2) {\pi \over 2} + \delta \}$.  To see this, suppose  $\Gamma \subset \mathbb{R}^{n}$ is a symmetric cone containing the set $E$.  We can find a point $\lambda := (\lambda_1,\ldots,\lambda_{n-1}, \lambda_n) \in E$ so that $\lambda_1 \geq \lambda_2 \geq \cdots \geq \lambda_{n-1} >0 > \lambda_n$.  We claim that $-e^{-Af}$ is not concave at $t\cdot \lambda$ for $t \gg0$.  Since  $\lambda_1,\ldots,\lambda_{n-1} >0$, arguing as in the proof of Lemma~\ref{Concave_G}, it suffices to evaluate the sign of the following determinant;
\[
\det\bigg(\frac{A+2t\lambda_i\delta_{ij}}{(1+t^2\lambda_i^2)(1+t^2\lambda_j^2)} \bigg)= \frac{2^{n-1}t^{n}\sigma_n(\lambda)}{\Pi(1+t^{2}\lambda_i^2)^2}\left(t^{-1}A\sum\frac{1}{\lambda_i}+2\right).
\]
By taking $t$ sufficiently large the right hand side of the above equation is clearly negative since $\sigma_n(\lambda)<0$.

An immediate consequence of Lemma~\ref{Concave_G} is
\begin{cor}\label{cor:Gconcave}
Let $\Gamma := \{ M \in {\rm Sym}(n) : F(M) \geq (n-2)\frac{\pi}{2} + \delta \}$.  Then there exists $A := A(\delta)$ so that the operator
\[
G = -e^{-AF}
\]
is elliptic and concave on $\Gamma$.
\end{cor}

\section{The Dirichlet Problem}

\subsection{Properties of the Subsolution}
The following lemma, due to Sz\'ekelyhidi \cite{S1} in large generality, building on previous work of Guan \cite{G1}, will be needed in the a priori estimates. We use the standard notation $F^{ij} = {\p F \over \p u_{ij}}$. 
\begin{lem} \label{lem:sze}
{\rm (Sz\'ekelyhidi)} Suppose there exists a function $\underline{u}$ such that for each point $p \in \Omega$ and each index $i$ there holds
\begin{equation} \label{eq:C_subsoln}
\lim_{t \rightarrow \infty} f(\underline{\lambda}+te_i)> h(p),
\end{equation}
where $\underline{\lambda}$ are the eigenvalues of $D^2 \underline{u}(p)$, $e_i$ is the $i^{\mathrm{th}}$ standard basis vector, and $f(\lambda) = \sum \arctan \lambda_i$. Let $u$ satisfy $f(\lambda)=h$ with $h: \overline{\Omega} \rightarrow [(n-2){\pi \over 2} + \delta, \, n \frac{\pi}{2})$, and $\lambda$ the eigenvalues of $D^2 u$.
\par
There are constants $R_0, \kappa >0$ with the following property. If $|\lambda| \geq R_0$, then we either have
\[
\sum_{i,j} F^{ij}(D^2 u) [\underline{u}_{ij}-u_{ij}] > \kappa \sum_p F^{pp}(D^2 u),
\]
or $F^{ii}(D^2 u)> \kappa \sum_p F^{pp}(D^2 u)$ for each $i$.
\end{lem}
\begin{proof}
For any $x \in \Omega$, let $\Gamma^{h(x)} = \{ \mu \in \mathbb{R}^n : f(\mu) > h(x) \}$. By Lemma~\ref{lem:cone_facts}, $\Gamma^{h(x)}$ is convex. Given the convexity of the level sets $\Gamma^{h(x)}$, the arguments in \cite{S1} (Proposition 6, Remark 8) go through verbatim.
\end{proof}
\par
In our case, we have
\begin{cor}\label{cor:subsol}
Suppose $\underline{u}$ is a subsolution satisfying \eqref{sub_def}, and let $u$ satisfy \eqref{eq:DirProb} with $h: \overline{\Omega} \rightarrow [(n-2){\pi \over 2} + \delta, \, n \frac{\pi}{2})$. As before, let $\lambda$ denote the eigenvalues of $D^2 u$. There exists $R_0$ depending only on $\underline{u}$ and $\delta$, such that for any $|\lambda|\geq R_0$, we have
\[
F^{ij}(D^2 u) [\underline{u}_{ij}-u_{ij}] \geq \tau>0,
\]
where $\tau$ is a constant depending on only $\underline{u}$ and $\delta$.
\end{cor}
\begin{proof}
Let $\underline{\lambda}$ be the eigenvalues of $D^2 \underline{u}(p)$. Then
\[
\lim_{t \rightarrow \infty} f(\underline{\lambda}+te_i) = \sum_{p \neq i} \arctan \underline{\lambda}_p + {\pi \over 2} > \sum_p \arctan \underline{\lambda}_p \geq h(p),
\]
which verifies \eqref{eq:C_subsoln}. Since the desired inequality is independent of coordinates, we choose coordinates such that $D^2 u$ is diagonal. It is well-known that in this case $F^{ij}=f_i \delta_{ij}$. 
\par
By Lemma~\ref{lem:cone_facts} and $\sup_{\overline{\Omega}} h <n \frac{\pi}{2}$, we have an estimate for the smallest eigenvalue $|\lambda_n|\leq C$. This allows us to rule out the case $F^{ii}(D^2 u)> \kappa \sum_p F^{pp}(D^2 u)$ for each $i$ in Lemma~\ref{lem:sze}.  Indeed, for $|\lambda| \geq R_0$ large enough, the largest eigenvalue $\lambda_1 \gg 1$ can be made arbitrarily large, and we have
\[
F^{11}={1 \over 1 + \lambda_1^2} \leq \kappa {1 \over 1 + \lambda_n^2} \leq \kappa \sum_p F^{pp},
\]
which rules out the second case of Lemma~\ref{lem:sze}. By the first case, we have
\[
F^{ij}(D^2 u) [\underline{u}_{ij}-u_{ij}] \geq \kappa {1 \over 1+\lambda_n^2} \geq \tau > 0,
\]
since $|\lambda_n|$ is bounded.
\end{proof}

\subsection{Zeroth and First Order Estimates}
From the existence of a subsolution it is straightforward to deduce two sided bounds for $u$.  Namely, we have
\begin{lem}\label{lem:c0 estimate}
Suppose $\underline{u}$ is a subsolution satisfying \eqref{sub_def}, and let $u$ satisfy \eqref{eq:DirProb}. Let $w:\Omega \rightarrow \mathbb{R}$ be the harmonic function defined by $\Delta w=0$ in $\Omega$, and $w|_{\del \Omega} = \phi$.  Then we have
\[
\underline{u} \leq u \leq w,
\]
and $\underline{u} = u = w=\phi$ on $\del \Omega$.
\end{lem}
\begin{proof}
By Lemma~\ref{lem:cone_facts} we know that $\Delta u >0$.  Thus the lemma follows from the maximum principle.
\end{proof}
Next, we derive an a priori gradient estimate.
\begin{prop} \label{prop:c1_est}
Suppose $\underline{u}$ is a $C^2(\overline{\Omega})$ subsolution satisfying \eqref{sub_def}, and let $u \in C^3(\overline{\Omega})$ satisfy \eqref{eq:DirProb} with $h: \overline{\Omega} \rightarrow [(n-2){\pi \over 2} + \delta, \, n \frac{\pi}{2})$. Then we have following estimate
\[
\sup_{\overline{\Omega}} |Du| \leq C(\Omega,\|\underline{u}\|_{C^1(\overline{\Omega})},\| h \|_{C^1(\overline{\Omega})},\delta).
\]
\end{prop}
\begin{proof}
We apply the maximum principle to $Q = \pm D_k u + {B \over 2} |x|^2$, with $B>0$ to be chosen later and $k \in \{ 1, \dots, n \}$ fixed. Suppose that $Q$ attains its maximum at an interior point $x_0\in \Omega$. By an orthogonal transformation, we may assume that $D^2 u (x_0)$ is diagonal. It follows that $F^{ij}$ is diagonal at $x_0$, and $F^{ij} =\frac{1}{1+\lambda^2_i} \delta_{ij}$. By differentiating the equation,
\begin{equation} \label{eq:diff_once}
F^{ij} D_i D_j D_k u = D_k h.
\end{equation}
At $x_0$ we have
\[
F^{ij} D_i D_j \bigg( {B \over 2} |x|^2 \bigg) = B \sum_i {1 \over 1 + \lambda_i^2} \geq {B \over 1 + \lambda_n^2},
\]
where $\lambda_n$ is the smallest eigenvalue. By Lemma~\ref{lem:cone_facts}, $\lambda_n$ is bounded below, and since $h(x) < {n \pi \over 2}$ on $\overline{\Omega}$, it follows that $\lambda_n$ is bounded above. Hence for $B$ large enough, we have $F^{ij} D_i D_j Q (x_0) > 0$. This implies the maximum of $Q$ must be attained at the boundary. By Lemma~\ref{lem:c0 estimate}, the gradient of $u$ is bounded by the gradients of $\underline{u}$ and $w$ on the boundary $\p \Omega$. Hence the maximum of $Q$ is bounded on the boundary. We can therefore uniformly bound each component $D_k u$ of the gradient, which gives the $C^1$ estimate.

\end{proof}

\subsection{Second Order Estimate}
To obtain a $C^2$ estimate, we will make use of the concavity of the operator $G$ from Corollary~\ref{cor:Gconcave}. Recall that we defined
\[
G(D^2 u) = - e^{-A F(D^2 u)}.
\]
By Corollary~\ref{cor:Gconcave}, $G$ is concave for $A$ large enough depending only on $\delta = \inf_{\Omega}h-(n-2)\frac{\pi}{2} >0$. In other words, if we denote 
\[
G^{ij} = \frac{\del G}{\del u_{ij}}, \qquad G^{ij,k \ell} = \frac{\del^2 G}{ \del u_{ij} \del u_{k \ell}},
\]
then
\begin{equation}\label{G^{ijkl}}
G^{ij,k \ell} M_{ij} M_{k \ell} \leq 0,
\end{equation} 
for any symmetric matrix $M_{ij}$. Note that the Dirichlet problem~\eqref{eq:DirProb} is equivalent to following Dirichlet problem
\begin{equation}\label{eq:DiriG}
\begin{aligned}
G(D^2u)&=-e^{-Ah(x)} := \psi(x)\ {\rm in} \ \Omega,\\
u &= \phi \ \ \ {\rm on} \ \del \Omega.
\end{aligned}
\end{equation}
The proof of the $C^2$ estimate follows the lines of Guan \cite{G1,G2}.  Since our operator does not quite fit the structural conditions imposed in \cite{G1,G2} (see the discussion after Lemma~\ref{Concave_G}), and since some simplifications occur in our particular setting, we provide the complete argument.
\begin{prop} \label{prop:c2_int}
Suppose $\underline{u}$ is a $C^2(\overline{\Omega})$ subsolution satisfying \eqref{sub_def}, and let $u \in C^4(\overline{\Omega})$ satisfy \eqref{eq:DirProb} with $h: \overline{\Omega} \rightarrow [(n-2){\pi \over 2} + \delta, \, n \frac{\pi}{2})$. Then we have following estimate
\[
\sup_{\overline{\Omega}} |D^2u| \leq C(\Omega,\| h \|_{C^2(\overline{\Omega})},\delta)(1+\max_{\partial\Omega}|D^2u|).
\]
\end{prop}
\begin{proof}
Consider the quantity $\Delta u + \frac{B}{2}|x|^{2}$; we claim that this quantity does not achieve an interior maximum, provided $B$ is chosen sufficiently large (but universal).  First, note that since $D^{2}u \geq -C$ by Lemma~\ref{lem:cone_facts}, the proposition follows from this claim.  
By differentiating the equation twice, we have
\[
G^{ij} D_i D_j \Delta u = \Delta \psi- \sum_k G^{ij,rs}u_{ijk}u_{rsk} \geq \Delta \psi
\]
by the concavity of $G$.  Furthermore, we have
\[
G^{ij} D_i D_j \frac{B}{2}|x|^{2} = B \sum_{i}G^{ii}.
\]
Fixing a point $x$ and performing an orthogonal transformation so that $D^2u(x)$ is diagonal, we have
\[
\sum_i G^{ii} = A e^{-Ah} \sum_{i}\frac{1}{1+\lambda_i^2} \geq \delta'
\]
for a universal constant $\delta'$.  Here we have used Lemma~\ref{lem:cone_facts} to deduce that the smallest eigenvalue of $D^2u$ is bounded in absolute value.  In particular, for $B$ sufficiently large, and universal, we have that $G^{ij} D_i D_j (\Delta u + \frac{B}{2}|x|^{2})>0$, and so the maximum is achieved on the boundary.
\end{proof}

The goal for the remainder of this section is to derive the $C^2$ estimate at the boundary.

\begin{prop}
Suppose $\underline{u}$ is a $C^4(\overline{\Omega})$ subsolution satisfying \eqref{sub_def}, and let $u \in C^3(\overline{\Omega})$ satisfy \eqref{eq:DirProb} with $h: \overline{\Omega} \rightarrow [(n-2){\pi \over 2} + \delta, \, n \frac{\pi}{2})$. Then we have following estimate
\[
\sup_{\partial \Omega} |D^2u| \leq C(\Omega,\|\underline{u}\|_{C^4(\overline{\Omega})}, \| h \|_{C^2(\overline{\Omega})},\delta).
\]
\end{prop}

At any point $x_0\in\partial\Omega$, choose coordinates $x_1,x_2,\dots,x_n$ with origin at $x_0$ such that the positive $x_n$ axis is in the direction of the interior normal of $\partial\Omega$ at $0$.  Denote $x'=(x_1,x_2,\dots,x_{n-1})$. Near $0$, we may represent $\partial \Omega$ as a graph which satisfies
\begin{equation} \label{rho}
x_n=\rho(x')=\frac{1}{2}\sum_{\alpha,\beta<n}\rho_{\alpha\beta}(0)x_{\alpha}x_{\beta}+O(|x'|^3).
\end{equation}
Since
\[
(u-\underline{u})(x',\rho(x'))=0,
\]
we have
\[ \label{bdd_id_1}
(u-\underline{u})_{x_{\alpha}x_{\beta}}(0)=-(u-\underline{u})_{x_n}(0)\rho_{\alpha\beta}(0), \ \ \text{ for}\ \alpha,\beta<n.
\]
From the boundary gradient estimate it follows that
\[
|u_{x_\alpha x_{\beta}}(0)| \leq C, \ \ \alpha,\beta<n.
\]
Next we will estimate the mixed normal-tangential derivatives, $u_{x_{\alpha}x_n}(0)$ for $\alpha < n$.  For this we will use a barrier argument exploiting the barrier function from \cite{G3},
\[
v = (u-\underline{u})+t d - N d^2,
\]
where $d(x) = d(x,\del \Omega)$ is the distance function to the boundary. We denote $L = F^{ij} D_i D_j$.  We need the following lemma:

\begin{lem} \label{Lv_lemma}
For $\delta'$ small enough, there exists $\epsilon_1>0$ depending on $\underline{u}$, $h$, $\Omega$ such that
\[
\begin{aligned}
L v&\leq -\epsilon_1, \ \ \text{ inside } \ \Omega \cap B_{\delta'}(0), \\
v&\geq 0, \ \ \text{ on} \ \partial (\Omega  \cap B_{\delta'}(0)).
\end{aligned}
\]
\end{lem}
\begin{proof}
First, we calculate inside $\Omega \cap B_{\delta'}(0)$,
\[
Lv = F^{ij} (u_{ij}-\underline{u}_{ij}) +t F^{ij} D_iD_j d  -2N d F^{ij} D_iD_j d-2N F^{ij}D_idD_jd.
\]
We consider two cases.  First, assume that $|\lambda |\leq R_0$, where $R_0$ is the constant from Corollary~\ref{cor:subsol}.  Since $d$ is the distance function, we know that $|\nabla d|=1$, and so we have
\[
F^{ij} D_id D_jd \geq \frac{1}{1+R_0^2}.
\]
Since $\del \Omega$ is a smooth hypersurface, we can assume that the distance function $d$ is smooth in $B_{\delta'}(0)\cap \del\Omega$,  \cite[Chapter 14]{GT}, and so
\[
\bigg| F^{ij}( u_{ij}- \underline{u}_{ij}) \bigg|\leq C, \ \ \bigg| F^{ij} D_i D_j d \bigg|\leq C, \ \ \bigg| d F^{ij} D_i D_j d \bigg|\leq C\delta'.
\]
Putting everything together we have
\[
Lv \leq C+tC+2NC\delta'-\frac{2N}{1+R_0^2}.
\]
Consider now the case when $\lambda > R_0$. By Corollary~\ref{cor:subsol}, we have 
\[
F^{ij}( u_{ij}- \underline{u}_{ij}) \leq -\tau.
\]
Hence
\[
Lv\leq -\tau+tC+2NC\delta'-2N F^{ij} D_i d D_j d.
\]
We fix constants as follows:
\[
N \gg C, \ \ t \ll \frac{\tau}{4C} \ \ \delta' \ll \frac{t}{NC}.
\]
Thus in either case we obtain
\[
L v\leq -\epsilon_1, \ \ \text{inside} \ \Omega \cap B_{\delta'}(0).
\]
Also we have 
\[
v\geq 0, \ \ {\mathrm on} \ \partial (\Omega  \cap B_{\delta'}(0)),
\]
since $v=(u-\underline{u})+(t-Nd)d$ and $t>N\delta'>Nd$. This proves the lemma.
\end{proof}

With this lemma in hand we can now estimate $u_{x_{\alpha}x_n}(0)$. For $\alpha \in \{1, \dots, n-1 \}$, define
\[
T_\alpha = \frac{\partial}{\partial x_{\alpha}}+ \sum_{\beta<n} \rho_{\alpha\beta}(0) \, \bigg( x_\beta \frac{\partial}{\partial x_n} - x_n \frac{\partial}{\partial x_\beta} \bigg).
\]
The vector field $T_\alpha$ is an approximate tangential operator on $\del \Omega$. Indeed, on $\del \Omega$, the operator $\frac{\partial}{\partial x_{\alpha}}+ \del_\alpha \rho \, \frac{\partial}{\partial x_n}$ is a tangential operator, and hence using (\ref{rho}), on $\del \Omega$ near $0$ we have
\[ \label{T_approx}
T_\alpha = [ \del_\alpha  + \del_\alpha \rho \, \del_n  ] + O(|x'|^2)\del_n -\sum_{\beta<n}\rho_{\alpha\beta}(0)\rho(x')\del_{\beta}
\]
on $\del \Omega$ near $0$.  Since $u=\underline{u}$ on $\del \Omega$, the boundary gradient estimate implies
\begin{equation}\label{T_approx}
\limsup_{(x',x_n)\rightarrow (x',\rho(x'))}|T_{\alpha}(u-\underline{u})| \leq C|x'|^{2}
\end{equation}
for a universal constant $C$.

The advantage of working with $T_\alpha$ is that the vector field $X_\beta = x_\beta \frac{\partial}{\partial x_n} - x_n \frac{\partial}{\partial x_\beta}$ generates a rotation. Since $F(D^2 u)$ only depends on the eigenvalues of the Hessian of $u$, it is invariant under rotations of coordinates. It follows that applying the vector field $X_\beta$ to the equation $F(D^2u) = h$ gives $X_\beta h = F^{ij} D_i D_j (X_\beta u)$, from which it follows that $T_\alpha h = L (T_\alpha u)$. Thus 

\begin{equation}\label{eq:LTest}
|L \, T_\alpha (u-\underline{u})|\leq C(1 + \sum F^{ii}) \leq C, \quad \text{ in }\Omega \cap B_{\delta'}(0).
\end{equation}
Choosing $\delta'$ as in Lemma \ref{Lv_lemma}, we can choose constants $A \gg B \gg 1$ large enough so that
\[
\begin{aligned}
L(Av+B|x|^2\pm T_\alpha(u-\underline{u}))&< 0, \ \mathrm {in}\ \Omega \cap B_{\delta'}(0),\\
\liminf_{x \rightarrow \del(\Omega \cap B_{\delta'}(0))} Av+B|x|^2\pm T_\alpha(u-\underline{u}) &\geq 0. 
\end{aligned}
\]
We choose the constants as follows.  First, since $v \geq 0$ on $\del(\Omega \cap B_{\delta'}(0))$, it suffices to choose $B$ large so that
\[
\liminf_{x \rightarrow \del(\Omega \cap B_{\delta'}(0))} B|x|^2\pm T_\alpha(u-\underline{u}) \geq 0.
\]
On $\Omega \cap \partial B_{\delta'}(0)$ we use $|x|=\delta'$, and that $ T_\alpha(u-\underline{u})$ is bounded by the gradient estimate.  On $\del \Omega \cap B_\delta' (0)$, we can choose $B$ large, and universal by the estimate \eqref{T_approx}.  Having chosen $B$ we choose $A$ using~\eqref{eq:LTest} and Lemma~\ref{Lv_lemma}. It follows that $Av+B|x|^2\pm T_\alpha (u-\underline{u}) \geq 0$ inside $\Omega\cap B_{\delta'}(0)$. Since $Av+B|x|^2\pm T_\alpha (u-\underline{u})$ attains zero at the origin, it follows that
\[
\del_n (Av+B|x|^2\pm T_\alpha (u-\underline{u})) (0) \geq 0,
\]

and so
\[
|u_{\alpha n}(0)-\underline{u}_{\alpha n}(0)|\leq |Av_n(0)| + |\sum_{\beta < n} \rho_{\alpha \beta} (u-\underline{u})_\beta(0)| \leq C,
\]
which gives the mixed second derivative bounds $|u_{\alpha n}| \leq C$ for all $\alpha < n$.
\bigskip \par
The next step is to estimate $u_{nn}$ on the boundary $\partial\Omega$. Recall that it suffices to obtain an upper bound $u_{nn} \leq C$, since $D^{2}u \geq -C$ by Lemma~\ref{lem:cone_facts}. We use an idea of N. Trudinger \cite{T1}, later used by B. Guan \cite{G1}, to obtain this estimate.
\par
Let us explain the main idea.  Fix a point $x \in \del \Omega$, which we assume to be the origin for simplicity, and let $\{e_i \}_{i=1}^{n}$ be an orthonormal local frame defined in a neighbourhood of the origin such that $e_n$ is the inner normal when restricted to $\del \Omega$. For $1 \leq \alpha,\beta \leq n-1$, define $\sigma_{\alpha\beta} = \langle \nabla_{e_{\alpha}} e_\beta, e_n \rangle$, where $\nabla$ denotes the covariant derivative with respect to the flat Euclidean metric. On $\del \Omega$, $\sigma_{\alpha \beta}$ is the second fundamental form and since $u|_{\del \Omega} = \underline{u}|_{\del\Omega}$ we see that for any $x \in \del \Omega$ there holds
\begin{equation} \label{bdd_identity}
u_{\alpha\beta}(x)-\underline{u}_{\alpha\beta}(x)=-(u-\underline{u})_n(x) \sigma_{\alpha\beta}(x),
\end{equation}
where
\[
u_{\alpha\beta} = \nabla_{e_\beta}(\nabla_{e_{\alpha}}u) - \nabla_{\nabla_{e_{\beta}}e_{\alpha}}u
\]
is the Riemannian Hessian.  Let us denote by $\lambda'(u_{\alpha\beta})$ the eigenvalues of the $(n-1)\times(n-1)$ matrix $u_{\alpha\beta}$; note that this is well-defined since the frame $\{e_i\}_{1\leq i\leq n}$ is assumed to be orthonormal. Recall $g(\lambda)=-e^{-A \sum \arctan \lambda_i}$ and $\psi(x)=-e^{-Ah(x)}$. Our goal is to prove the following lemma

\begin{lem}\label{lem:bdC2main}  
There exist constants $R_0, c_0 >0$ such that for all $R \geq R_0$ there holds
\[
g(\lambda'(u_{\alpha \beta}), R) > \psi(x)+c_0, \qquad \forall x \in \del\Omega.
\]
\end{lem}

Let us explain how this lemma implies the boundary $C^2$ estimate.  Fix a point $p \in \del \Omega$.  Fix coordinates $ (x_1,\ldots,x_{n-1},x_n)$ near $p$ so that $p$ is the origin and $\frac{\del}{\del x_i} = e_{i}(p)$; in particular, $\frac{\del}{\del x_n}$ is the interior normal for $\del \Omega$ at $p$.  By an orthogonal transformation, we may assume that $u_{\alpha\beta}$ is diagonal at $x$.  We need the following lemma, due to Caffarelli-Nirenberg-Spruck
\begin{lem}[\cite{CNS3}, Lemma 1.2]\label{lem:CNSLem}
Consider the $n\times n$ symmetric matrix
\[
\begin{pmatrix}
d_1 &0 & \cdots&0&a_1\\
0&d_2&\cdots&0&a_2\\
\vdots & \vdots& \ddots& \vdots&\vdots\\
0 &0&\,&d_{n-1}&a_{n-1}\\
a_1&a_2&\cdots&a_{n-1}&a
\end{pmatrix}
\]
with $d_1,\ldots,d_{n-1}$ fixed, $|a|$ tending to infinity, and $|a_i|<C$ for $1\leq i\leq n-1$. Then the eigenvalues $\lambda_1,\ldots,\lambda_n$ behave like
\[
\begin{aligned}
\lambda_{\alpha} &= d_{\alpha} +o(1) \qquad 1\leq \alpha\leq n-1\\
\lambda_n &= a\left( 1+ O\left(\frac{1}{a}\right)\right)
\end{aligned}
\]
where $o(1)$ and $O\left(\frac{1}{a}\right)$ are uniform depending only on $d_1,\ldots,d_{n-1}, C$.
\end{lem}
As a consequence of this lemma, for every $\delta_0>0$ there exists $R_{\delta_0} \gg 1$ such that if $u_{nn}(0) \geq R_{\delta_0}$, then the eigenvalues of $u_{ij}(0)$ satisfy
\[
|\lambda(u_{ij}) - (\lambda'(u_{\alpha\beta}),u_{nn})| < \delta_0.
\]
By continuity of $g$, there exists a $\delta_0>0$ depending on $c_0$ such that if $u_{nn}(0) \geq \max \{ R_{\delta_0},R_0\}$, then
\[
\begin{aligned}
g(\lambda(u_{ij}))(0) &\geq g(\lambda'(u_{\alpha\beta}), u_{nn})(0)-\frac{c_0}{2} \\
&\geq g(\lambda'(u_{\alpha\beta})(0), R_0) - \frac{c_0}{2}\\
&>\psi(0) +\frac{c_0}{2},
\end{aligned}
\]
which is a contradiction.  Thus, it suffices to prove Lemma~\ref{lem:bdC2main}, 
which will be the goal of the remainder of this section.  To begin, for $x \in \del \Omega$, we define operators 
\[
\begin{aligned}
\tilde{G}(u_{\alpha\beta}(x))&=\lim_{R\rightarrow\infty}g(\lambda'(u_{\alpha\beta}),R) = - \exp \bigg(-A \sum_{\alpha=1}^{n-1} \arctan \lambda'_\alpha - A {\pi \over 2} \bigg), \\
\tilde{G}_0^{\alpha\beta}&=\frac{\partial\tilde{G}}{\partial u_{\alpha\beta}}(u_{\alpha\beta}(0)).
\end{aligned}
\]

Consider the function $\tilde{G}(u_{\alpha\beta})(x)-\psi(x)$ on $\del\Omega$.  Assume the minimum of this function is achieved at $x_0 \in \del \Omega$. Choose coordinates such that $x_0$ is the origin and the $(n-1) \times (n-1)$ upper block $u_{\alpha \beta}(0)$ is diagonal, with $x_n$ the inner normal of $ \Omega$ at $0$. We claim that to obtain Lemma~\ref{lem:bdC2main} it suffices to obtain an upper bound $u_{nn}(0) \leq C$. Indeed, by Lemma~\ref{schur_horn_u} below, in this case we have
\[
\sum_{\alpha=1}^{n-1} \arctan \lambda'_\alpha + {\pi \over 2} \geq h(0) - \arctan u_{nn} + {\pi \over 2} \geq h(0) + \epsilon.
\]
Then
\[
\tilde{G}(0) \geq -e^{-A h(0)} e^{-A \varepsilon} = \psi(0) + e^{-A h(0)}(1 - e^{-A \varepsilon}),
\]
which proves a lower bound $\tilde{G}(x) - \psi(x) \geq 2 c_0 >0$. We may now choose $R_0$ large enough such that
\[
g(\lambda'(u_{\alpha\beta}),R_0) \geq \tilde{G}(x) - c_0 \geq \psi(x) + c_0,
\]
and Lemma~\ref{lem:bdC2main} follows.  Thus, it suffices to prove the estimate $u_{nn}(0) \leq C$ where $0 \in \del \Omega$ is the point where  $\tilde{G}(u_{\alpha\beta})(x)-\psi(x)$ achieves its minimum value.

The main property of $\tilde{G}$ is that it is a concave function of $u_{\alpha \beta}$. This follows from the following lemma, together with Lemma~\ref{Concave_G}.  
\begin{lem}\label{lem:n-3}
At any $x \in \del \Omega$, there holds
\[ \label{schur_horn_u}
\sum_{\alpha=1}^{n-1} \arctan \lambda'_\alpha \geq \sum_{i=1}^{n} \arctan \lambda_i - \arctan u_{nn}(x).
\]
It follows that
\[
\sum_{\alpha=1}^{n-1} \arctan \lambda'_{\alpha} \geq {(n-3) \pi \over 2}+\delta.
\]
\end{lem}
\begin{proof}
Fix $x \in \del \Omega$.  By performing an orthogonal transformation to the vector fields $\{e_{\alpha}\}_{1\leq \alpha \leq n-1}$, we may assume that the $(n-1) \times (n-1)$ upper block $u_{\alpha \beta}(x)$ is diagonal, with $e_n$ the inner normal of $\del \Omega$ at $x$.
\[
\nabla^2u= \left( \begin{array}{cccc}
  \lambda'_1 & 0 & 0 & \ast \\
  0 &  \lambda'_2 & 0& \ast \\
  0 & 0 & \ddots & \ast\\
  \ast & \ast & \ast & u_{nn}\\
  \end{array} \right).
\]
By the Schur-Horn theorem \cite{H1}, $( \lambda'_1,\dots,\lambda'_{n-1},u_{nn} )$ is in the convex hull of vectors which are permutations of $( \lambda_1, \dots, \lambda_n )$. Let $\Gamma^{h(x)}$ be the set of $\mu \in \mathbb{R}^n$ such that $\sum_i \arctan \mu_i \geq h(x)$. By Lemma~\ref{lem:cone_facts}, $\Gamma^{h(x)}$ is convex. Since any permutation of $( \lambda_1, \dots, \lambda_n )$ lies in $\Gamma^{h(x)}$, we see that $(\lambda'_1,\dots,\lambda'_{n-1},u_{nn} )$ lies in $\Gamma^{h(x)}$. The lemma follows.
\end{proof}

As mentioned above, this lemma combined with Lemma~\ref{Concave_G} implies

\begin{cor}\label{cor:tGconcave}
The operator $\tilde{G}$ is concave on the set $u_{\alpha\beta}(\del \Omega)$.
\end{cor}

From the corollary we can deduce a version of Corollary~\ref{cor:subsol} on $\del \Omega$.
\begin{lem} \label{G_0_tau_lemma}
For $u_{nn}(0)$ large enough, there exists $\tau>0$ such that
\[
\tilde{G}_0^{\alpha\beta}(\underline{u}_{\alpha\beta}-u_{\alpha\beta})(0)\geq\tau>0.
\] 
\end{lem}
\begin{proof}
Denote $\underline{\lambda}_i$ eigenvalues of $\underline{u}_{ij}$, $1\leq i, j \leq n$ and $\underline{\lambda}'_{\alpha}$ eigenvalues of $\underline{u}_{\alpha\beta}$, $1\leq \alpha, \beta \leq n-1$. By the same argument as in Lemma~\ref{lem:n-3} we have
\begin{equation} \label{schur_horn_applied}
\arctan\underline{\lambda}'_1+\cdots+\arctan\underline{\lambda}'_{n-1}+\arctan \underline{u}_{nn}\geq \sum\arctan\underline{\lambda}_i. 
\end{equation}
On the other hand, by Lemma~\ref{lem:CNSLem} we know that for any $\epsilon>0$, if $u_{nn}(0)$ is large enough we can ensure
\begin{equation} \label{lambda'2lambda}
\sum_{\alpha \leq n-1} \arctan \lambda'_{\alpha} \leq \sum_{\alpha\leq n-1}\arctan \lambda_{\alpha}+\frac{\epsilon}{2}.
\end{equation}
Hence by ~\eqref{schur_horn_applied} and ~\eqref{lambda'2lambda}, we have
\[
\begin{aligned}
\sum_{\alpha\leq n-1}\arctan\underline{\lambda}'_{\alpha}-\sum_{\alpha\leq n-1}\arctan\lambda'_{\alpha} &\geq \sum_{1\leq i\leq n}\arctan\underline{\lambda}_i -\arctan\underline{u}_{nn}\\
&\quad -\sum_{\alpha \leq n-1} \arctan \lambda_{\alpha}-\frac{\epsilon}{2}  \\
&\geq h(x)-(h(x)-\arctan\lambda_n)-\arctan\underline{u}_{nn}-\frac{\epsilon}{2}  \\
&\geq \arctan\lambda_n-\arctan\underline{u}_{nn}-\frac{\epsilon}{2}  \\
&\geq \arctan u_{nn}-\arctan\underline{u}_{nn}-\epsilon.
\end{aligned}
\]
Assuming $u_{nn}(0)$ is large enough, there exists $c_1>0$ such that at the origin
\begin{equation} \label{n-1_subsoln}
\sum_{\alpha\leq n-1} \arctan\underline{\lambda}'_{\alpha}\geq\sum_{\alpha\leq n-1} \arctan \lambda'_{\alpha}+c_1.
\end{equation}
By concavity,
\[
\tilde{G}_0^{\alpha \beta} (\underline{u}_{\alpha\beta}-u_{\alpha\beta})(0)\geq \tilde{G}(\underline{u}_{\alpha\beta}(0)) - \tilde{G}(u_{\alpha\beta}(0)).
\]
Hence ~\eqref{n-1_subsoln} and the double tangential estimate yield a $\tau>0$ such that
\[
\tilde{G}_0^{\alpha \beta} (\underline{u}_{\alpha\beta}-u_{\alpha\beta})(0)\geq  \exp \bigg(-A \sum_{\alpha=1}^{n-1} \arctan \lambda'_\alpha - A {\pi \over 2} \bigg) (1-e^{-Ac_1}) \geq \tau.
\]

\end{proof}

We can now construct the test function for the boundary $C^2$ estimate.  Recall that $0\in \del \Omega$ is a point where $\tilde{G}(u_{\alpha\beta})(x)-\psi(x)$ achieves its minimum value.  By the preceding discussion we may assume that $u_{nn}(0) \gg 1$, otherwise we are done. Let $\eta := \tilde{G}_0^{\alpha\beta}\sigma_{\alpha\beta}$. Contracting (\ref{bdd_identity}) with $\tilde{G}_0^{\alpha \beta}$ and applying Lemma \ref{G_0_tau_lemma}, we see there exists $\epsilon$ depending on the $C^1$ estimate, and a constant $c_{\tau}$ depending only on $\tau$ from Lemma~\ref{G_0_tau_lemma} such that
\[
\eta(0)\geq \frac{\tau}{(u-\underline{u})_n(0)}\geq 2\epsilon c_\tau.
\]
By the tangential $C^2$ estimate, $\tilde{G}_0^{\alpha\beta}$ is uniformly elliptic, with universally controlled eigenvalues.  In particular, there exists a universal constant $\delta'>0$ such that $\eta \geq \epsilon c_\tau$ in a small neighbourhood $B_{\delta'}(0)\cap \Omega$. We construct 
\[
\Phi =-(u-\underline{u})_n+\frac{1}{\eta}\tilde{G}_0^{\alpha\beta}(\underline{u}_{\alpha\beta}(x)-u_{\alpha\beta}(0))-\frac{\psi(x)-\psi(0)}{\eta}.
\]
On the boundary $\partial\Omega$, by (\ref{bdd_identity})
\[
\begin{aligned}
\Phi &=\frac{1}{\eta}\{-(u-\underline{u})_n(x)\tilde{G}_0^{\alpha\beta}\sigma_{\alpha\beta}+\tilde{G}_0^{\alpha\beta}\underline{u}_{\alpha\beta}(x)-\tilde{G}_0^{\alpha\beta}u_{\alpha\beta}(0) \} -\frac{\psi(x)-\psi(0)}{\eta} \\
&=\frac{1}{\eta}\{ \tilde{G}_0^{\alpha\beta}(u_{\alpha\beta}(x)-u_{\alpha\beta}(0)) \}-\frac{\psi(x)-\psi(0)}{\eta} \\
&\geq  \frac{1}{\eta}\{ \tilde{G}(u_{\alpha\beta}(x))-\tilde{G}(u_{\alpha\beta}(0)) \}-\frac{\psi(x)-\psi(0)}{\eta} \\
&=\frac{1}{\eta} \{ \tilde{G}(u_{\alpha\beta}(x))-\psi(x)-(\tilde{G}(u_{\alpha\beta}(0))-\psi(0)) \}  \\
&\geq  0.
\end{aligned}
\]
The first inequality follows from the concavity of $\tilde{G}$, (see Corollary~\ref{cor:tGconcave}).  On $\del B_{\delta'}(0) \cap \Omega$, we have $\Phi \geq -C$ by the bound for the tangential second derivatives of $u$. We compute
\begin{equation}\label{eq:LPhi_comp}
\begin{aligned}
L \Phi &= - h_n + F^{ij} \underline{u}_{nij} +\frac{1}{\eta}\tilde{G}_0^{\alpha\beta} F^{ij} \underline{u}_{\alpha\beta i j} -\frac{2}{\eta^2}\tilde{G}_0^{\alpha\beta} F^{ij} \underline{u}_{\alpha\beta i} \eta_j \\
&+ F^{ij} D_i D_j \bigg(\frac{1}{\eta}\bigg)\tilde{G}_0^{\alpha\beta}(\underline{u}_{\alpha\beta}(x)-u_{\alpha\beta}(0))  - F^{ij} D_i D_j \bigg( \frac{\psi-\psi(0)}{\eta} \bigg).
\end{aligned}
\end{equation}
Thus we have $L \Phi \leq C (1 + \sum_i F^{ii}) \leq C$. As in the estimate for mix-tangential derivative, we can again use Lemma \ref{Lv_lemma} and take $A \gg B\gg 1$ large enough to obtain
\[
\begin{aligned}
L(Av+B|x|^2+\Phi)&\leq 0\ \mathrm {in} \ B_{\delta}(0)\cap \Omega \nonumber\\
 Av+B|x|^2+\Phi &\geq 0 \ \mathrm {on} \ \partial (B_{\delta'}(0)\cap \Omega).
\end{aligned}
\]
By the maximum principle, $Av+B|x|^2+\Phi\geq 0$ in $\ B_{\delta'}(0)\cap \Omega$. Thus $\Phi_n(0)\geq-(Av+B|x|^2)_n\geq -C$, which gives $u_{nn}(0) \leq C$. As previously explained, this yields Lemma~\ref{lem:bdC2main}, and hence the boundary $C^2$ estimate is complete.

\subsection{Solving the Equation} \label{cont_meth_sect}
We have established a priori estimates up to $C^2$ of solutions to the Dirichlet problem ~\eqref{eq:DirProb}. By the Evans-Krylov theorem ~\cite{E,K}, since $u$ also solves equation ~\eqref{eq:DiriG} which involves the concave operator $G$, we have control of the H\"older continuity of the second derivatives. Differentiating the equation and applying the Schauder estimates gives us
\begin{equation} \label{eq:C2alpha_est}
\| u \|_{C^{3,\alpha}(\overline{\Omega})} \leq  C(\Omega,\|\underline{u}\|_{C^4(\overline{\Omega})}, \| h \|_{C^2(\overline{\Omega})},\delta),
\end{equation}
for any $u \in C^4(\overline{\Omega})$ solving ~\eqref{eq:DirProb}, and any $0<\alpha<1$. 
\par
To solve the equation, we use the continuity method. Suppose $\underline{u}$ is a $C^4(\overline{\Omega})$ subsolution satisfying $F(D^2 \underline{u}) \geq h, \ \underline{u}|_{\p \Omega}=\phi,$ with $h: \overline{\Omega} \rightarrow [(n-2){\pi \over 2} + \delta, \, n \frac{\pi}{2})$ in $C^{2,\alpha}(\overline{\Omega})$. Consider the family of equations
\begin{equation}\label{eq:ContMeth}
\begin{aligned}
F(D^2u_t)&= th + (1-t) h_0 \ {\rm in} \ \Omega,\\
u_t &= \phi \ \ \ {\rm on} \ \del \Omega,
\end{aligned}
\end{equation}
where $h_0= F(D^2 \underline{u})$. Let
\[
S = \{ t \in [0,1] : {\rm there} \ {\rm exists} \ u_t \in C^{4,\alpha}(\overline{\Omega}) \ {\rm solving} \ ~\eqref{eq:ContMeth} \}.
\]
We have $0 \in S$ by taking $u_0 = \underline{u}$. The fact that $S$ is open follows from the invertibility of the linearized operator and the implicit function theorem. That $S$ is closed follows from the a priori estimates. Indeed, the subsolution is preserved along the path, and the right-hand side of ~\eqref{eq:ContMeth} stays greater than $(n-2)\frac{\pi}{2} + \delta$. If $u_t \in C^{4,\alpha}(\overline{\Omega})$ solves ~\eqref{eq:ContMeth}, may apply estimate ~\eqref{eq:C2alpha_est} and we see that $S$ is closed. Hence $S=[0,1]$.
\par
It follows that there exists a smooth solution $u \in C^{\infty}(\overline{\Omega})$ to the Dirichlet problem ~\eqref{eq:DirProb} if all data is smooth, since we may differentiate the equation and apply Schauder theory. Uniqueness of the solution follows from the maximum principle for fully nonlinear PDE. If the right-hand side $h \in C^2(\overline{\Omega})$, we may take a sequence of smooth right-hand sides $h_\epsilon$ approximating $h$, obtain a sequence of solutions $u_\epsilon$, and apply estimate \eqref{eq:C2alpha_est} and a limiting process to solve the equation.

\section{The Complex Dirichlet Problem} \label{cplx_diri}
\subsection{Preliminary Estimates}
The techniques from the previous section can be adapted to solve the complex Dirichlet problem
\begin{equation} \label{eq:cplx_diri}
F(\p \bar{\p} u) := \sum_i \arctan \lambda_i = h(z), \ \ u|_{\p \Omega} = \phi
\end{equation}
in a domain $\Omega \subset \mathbb{C}^n$, where $\lambda$ denotes the eigenvalues $\lambda_1, \dots, \lambda_n$ of the complex Hessian $u_{i \bar{j}} = \p_i \p_{\bar{j}} u$. We use the notation $\p_i = {1 \over 2} ( {\p \over \p x_i} - \sqrt{-1} {\p \over \p y_i})$, and $\p_{\bar{i}} = {1 \over 2} ( {\p \over \p x_i} + \sqrt{-1} {\p \over \p y_i})$.
\par
The linearized operator becomes $L = F^{i \bar{j}} \p_i \p_{\bar{j}}$, where $F^{i \bar{j}} = {\p F \over \p u_{i \bar{j}}}$. We have $\underline{u} \leq u \leq w$ as in Lemma ~\ref{lem:c0 estimate}, and applying $L$ to $Q = \pm u_k + {B \over 2} |z|^2$ gives the gradient estimate as in Proposition ~\ref{prop:c1_est}. Hence
\[
\| u \|_{C^1(\overline{\Omega})} \leq C(\Omega,\|\underline{u}\|_{C^1(\overline{\Omega})},\| h \|_{C^1(\overline{\Omega})},\delta).
\]
For the $C^2$ estimate, we can again make use of the operator 
\[
G(\p \bar{\p} u) = g(\lambda) = -e^{-A \sum_i \arctan \lambda_i}
\]
which is concave on the space 
\[
\Gamma = \{ H \in {\rm Herm}(n) : F(H) \geq (n-2) {\pi \over 2} + \delta \},
\]
by Lemma \ref{Concave_G}. Here ${\rm Herm}(n)$ is the space of $n \times n$ Hermitian matrices. For any $\xi \in S^{2n-1}$, we let $D_\xi = \sum_{k=1}^n \xi^k {\p \over \p x_k} + \sum_{k=n+1}^{2n} \xi^k {\p \over \p y_k}$. Suppose
\[
\sup_{(\xi,z) \in S^{2n-1} \times \overline{\Omega}} \bigg( D_\xi D_\xi u + {B \over 2}|z|^2 \bigg)
\]
attains its maximum at an interior point $z_0 \in \Omega$ in the direction $\xi_0$. We may choose complex coordinates such that $\xi_0 = {\p \over \p x_1}$. By differentiating the equation twice
\[
G^{i \bar{j}} (D_{x_1} D_{x_1} u)_{i \bar{j}} = -G^{i \bar{j},k \bar{\ell}} D_{x_1} u_{i \bar{j}} D_{x_1} u_{k \bar{\ell}} + D_{x_1} D_{x_1} h.
\]
By concavity, $-G^{i \bar{j},k \bar{\ell}} D_{x_1} u_{i \bar{j}} D_{x_1} u_{k \bar{\ell}} \leq 0$. As in Proposition ~\ref{prop:c2_int}, we have
\[
G^{i \bar{j}} \bigg({B \over 2}|z|^2 \bigg)_{i \bar{j}} = B \sum_i G^{i \bar{i}} \geq \delta'>0,
\]
and hence by choosing $B \gg 1$ we can force $G^{i \bar{j}} (D_{x_1} D_{x_1} u + {B \over 2}|z|^2)_{i \bar{j}} > 0$ at $z_0$. By the maximum principle, for any $\xi \in S^{2n-1}$ we have
\[
(D_\xi D_\xi u + {B \over 2}|z|^2) \leq C(\Omega, \|h\|_{C^2(\overline{\Omega})},\delta)(1 + \max_{\p \Omega} |D^2 u|).
\]
Since $\Delta u \geq 0$, we can deduce the following estimate on the real Hessian of a solution $u$
\[
\sup_{\overline{\Omega}} |D^2 u| \leq C(\Omega, \|h\|_{C^2(\overline{\Omega})},\delta)(1 + \max_{\partial\Omega}|D^2u|).
\] 
It remains to estimate the real Hessian of $u$ at the boundary. 
\subsection{Boundary Mixed Tangential-Normal Estimate}
Fix a point $z_0 \in \p \Omega$ and choose coordinates such that $x_n$ is in the direction of the inner normal and locally $\p \Omega$ is given by
\[
x_n = \rho(t') = {1 \over 2} \sum_{\alpha,\beta < 2n} b_{\alpha \beta} t_\alpha t_\beta + O(|t'|^3),
\]
where we write $t_\alpha = y_\alpha$, for $1 \leq \alpha \leq n$ and $t_{\alpha+n}=x_\alpha$, for $1\leq \alpha \leq n-1$ and $t'=(t_1, \dots, t_{2n-1})$. Differentiating $(u-\underline{u})(t',\rho(t'))=0$ gives
\begin{equation} \label{eq:order_t'}
|(u-\underline{u})_{t_\beta}(t',\rho(t'))| \leq C|t'|, \ \ \beta < 2n,
\end{equation}
for a universal constant $C$.  As before, the double tangential derivatives are under control:
\[
|u_{t_\alpha t_\beta}(0)| \leq C, \ \ {\alpha,\beta < 2n}.
\]
We define $v(z)$ by
\[
v = (u-\underline{u})+t d - N d^2,
\]
 where $d(z) = d(z,\p \Omega)$. An analogous argument to Lemma ~\ref{Lv_lemma} (only difference is that the complex gradient $\nabla d = (d_1, \dots, d_n)$ now has norm ${1 \over 2}$) gives small $\delta', \ \epsilon >0$ such that
\begin{equation}
\begin{aligned} \label{eq:Lv_cplx}
L v&\leq -\epsilon_1, \ \ \text{ inside } \ \Omega \cap B_{\delta'}(0), \\
v&\geq 0, \ \ \text{ on} \ \partial (\Omega  \cap B_{\delta'}(0)).
\end{aligned}
\end{equation}
We next define the approximate tangential operator
\[
T_\alpha = \frac{\partial}{\partial t_{\alpha}}+ \sum_{\beta<2n} b_{\alpha\beta} \,  t_\beta \frac{\partial}{\partial x_n}.
\]
As in ~\eqref{T_approx}, we have
\begin{equation} \label{eq:T_approx_cplx}
|T_\alpha (u - \underline{u})| \leq C |t'|^2.
\end{equation}

\begin{lem} \label{lem:LTu} With the above notation, there is a universal constant C so that the following estimate holds:
\[
|L T_\alpha (u-\underline{u})| \leq C + F^{i \bar{j}} \p_i (u-\underline{u})_{y_n} \p_{\bar{j}} (u-\underline{u})_{y_n} .
\]
\end{lem}
\begin{proof}
The proof is adapted from the proof of \cite[Lemma 4.3]{LiSY}.  We include the brief computation for the reader's convenience.  First, we compute
\[
\begin{aligned}
L T_{\alpha} u &= T_{\alpha}h + \sum_{\beta <2n} b_{\alpha \beta} F^{i\bar{j}}\left(\del_{i}t_{\beta}\del_{\bar{j}}u_{x_n} + \del_{\bar{j}}t_{\beta}\del_{i}u_{x_{n}}\right)\\
&= T_{\alpha}h + 2\sum_{\beta<2n} b_{\alpha \beta}F^{i\bar{j}}\left(\del_{i}t_{\beta}\del_{\bar{j}}u_{n} + \del_{\bar{j}}t_{\beta}\del_{i}u_{\bar{n}}\right)\\
&\quad +\sqrt{-1}\sum_{\beta<2n}b_{\alpha \beta}F^{i\bar{j}}(\del_i t_\beta \del_{\bar{j}}u_{y_n} - \del_{\bar{j}}t_{\beta}\del_{i}u_{y_{n}}).
\end{aligned}
\]
By directly computing $\del_i t_{\beta}$ we obtain
\[
\begin{aligned}
LT_{\alpha}u &= T_{\alpha}h + 2\sum_{\beta=1}^{n}b_{\alpha\beta} {\rm Im}(F^{\beta\bar{j}}u_{n\bar{j}}) + 2\sum_{\beta=1}^{n-1}b_{\alpha\beta+n}{\rm Re}(F^{\beta\bar{j}}u_{n\bar{j}})\\
&\quad + \sqrt{-1}\sum_{\beta<2n}b_{\alpha \beta}F^{i\bar{j}}(\del_i t_\beta \del_{\bar{j}}u_{y_n} - \del_{\bar{j}}t_{\beta}\del_{i}u_{y_{n}}).
\end{aligned}
\]
Since the left-hand side of the equation is real, we must have
\[
{\rm Re}\left(\sum_{\beta<2n}b_{\alpha \beta}F^{i\bar{j}}(\del_i t_\beta \del_{\bar{j}}u_{y_n} - \del_{\bar{j}}t_{\beta}\del_{i}u_{y_{n}})\right)=0,
\]
which of course can be seen directly by inspection.  From the linearization of $F$, the eigenvalues of the Hermitian matrix $F^{i\bar{j}}u_{k\bar{j}}$ are bounded in absolute value by $1$.  To see this just recall that by choosing coordinates so that $u_{k\bar{j}}$ is diagonal with $u_{k\bar{k}} = \lambda_{k}$, we get that 
\begin{equation}\label{eq:LinbddEV}
F^{k\bar{j}} = \frac{\delta_{k\bar{j}} }{(1+\lambda_k^2)},
\end{equation} 
so that $F^{i\bar{j}}u_{k\bar{j}}$ has eigenvalues $\frac{\lambda_k}{1+\lambda_{k}^{2}}$.  It follows immediately that the components of the matrix $F^{i\bar{j}}u_{k\bar{j}}$ are bounded in norm by a constant depending only on the dimension.  By the same computation applied to $T_{\alpha}\underline{u}$, and using that $F^{i\bar{j}}\underline{u}_{k\bar{j}}$ has bounded eigenvalues by ~\eqref{eq:LinbddEV}, we obtain
\[
\left|LT_{\alpha}(u- \underline{u})\right| \leq C + \left|{\rm Im}\sum_{\beta<2n}b_{\alpha \beta}F^{i\bar{j}}(\del_i t_\beta \del_{\bar{j}}(u-\underline{u})_{y_n} - \del_{\bar{j}}t_{\beta}\del_{i}(u-\underline{u})_{y_{n}})\right|
\]
for a universal constant $C$.  The lemma follows from the Cauchy-Schwarz inequality.
\end{proof}
We consider
\[
\Psi = Av + B|z|^2 - (u_{y_n}-\underline{u}_{y_n})^2 \pm T_\alpha (u- \underline{u}),
\]
for constants $A,B$ to be determined.  Using ~\eqref{eq:order_t'} and ~\eqref{eq:T_approx_cplx} we see that on $\p \Omega \cap B_{\delta'}(0)$ we have
\[
\Psi \geq B|z|^{2} - C_1|z|^2
\]
for a universal constant $C_1$.  Taking $B \gg C_1$, we can ensure that $\Psi \geq 0$ on $\p \Omega \cap B_{\delta'}(0)$.  Furthermore, on $\p B_{\delta'}(0) \cap \Omega$, since all derivatives of $u$ are bounded and $B|z|^2=B\delta'^2$, we can obtain $\Psi \geq 0$ for $B$ large enough.  In particular, we can find a universal constant $B$ so that $\Psi \geq 0$ on $\p(\Omega \cap B_{\delta'}(0))$.  On the other hand, using ~\eqref{eq:Lv_cplx} and Lemma ~\ref{lem:LTu}, we compute
\[
L \Psi \leq -A \epsilon_1 +nB - 2 (u_{y_n}-\underline{u}_{y_n}) (h_{y_n}-F^{i \bar{j}}\underline{u}_{y_n i \bar{j}}) +C.
\]
Taking $A \gg B \gg 1$, we can arrange that $L\Psi \leq 0$. Hence
\[
L \Psi \leq 0, \ \ \Psi \geq 0 \ \mathrm {on} \ \partial (B_{\delta'}(0)\cap \Omega).
\]
This implies $\Psi \geq 0$ in $B_{\delta'}(0)\cap \Omega$. At the origin, we have $v(0) = (u_{y_n}-\underline{u}_{y_n})(0)=0$, hence $\Psi(0)=0$. Therefore
\[
\Psi_{x_n}(0)= Av_{x_n}(0)\pm u_{t_\alpha x_n}(0) \geq 0,
\]
which gives the mixed second derivative bounds $|u_{t_\alpha x_n}(0)| \leq C$ for all $\alpha < 2n$.
\bigskip \par
\subsection{Boundary Double Normal Estimate}
It remains to estimate $u_{x_n x_n}$ on $\p \Omega$.  The proof uses an argument of Guan-Sun \cite{GS}, which is based on an idea of Trudinger \cite{T1}, similar to the argument used in the real case.  As before, let $d(x) = d(x, \del \Omega)$ be the distance to the boundary, and define a vector bundle on $\del \Omega$ by
\[
T^{1,0}_{p}\del \Omega = \{ \xi \in T^{1,0}\mathbb{C}^{n} : \nabla_{\xi}d =0 \}.
\]
$T^{1,0}\del \Omega$ is a complex subbundle of $T^{1,0}\mathbb{C}^{n}|_{\p \Omega}$ of rank $n-1$.  For the sake of concreteness, let us give a description of the fiber of $T^{1,0}\del \Omega$  at a point $p \in \del \Omega$ in terms of local coordinates.  Choose a local coordinate $x_n$ so that $\frac{\del}{\del x_{n}}$ is the inward normal vector for $\del \Omega$ at $p$, and define a local coordinate $y_{n}$ so that $\frac{\del}{\del y_{n}} = J \frac{\del}{\del x_{n}}$ near $p$.  Then $z_{n} = x_n+\sqrt{-1}y_{n}$ defines a local holomorphic coordinate.  Complete this to a local holomorphic coordinate system by taking $z_{1}, \ldots, z_{n-1}$ so that $\frac{\del}{\del z_i}$ is orthogonal to $\frac{\del}{\del z_n}$ at $p$.  Then, at $p$ we have
\[
T^{1,0}_{p}\del \Omega = {\rm Span}_{\mathbb{C}}\left\{ \frac{\del}{\del z_1}, \ldots,\frac{\del}{\del z_{n-1}}\right\}.
\]
The above description makes it clear that, locally, near any point $p \in \del \Omega$, we can choose a smooth, orthonormal frame $\{\zeta_i\}_{i=1}^{n}$ of $T^{1,0}\mathbb{C}^{n}$ such that $\zeta_1, \dots, \zeta_{n-1}$ is an orthonormal frame in $T^{1,0}(\p \Omega)$, and ${\mathrm Re}(\zeta_n)$ parallel to the inner normal of $\p \Omega$ when restricted to $\del \Omega$.  Furthermore, at $p$ we may assume that $\zeta_i = \frac{\del}{\del z_i}$. Let $\sigma_{\alpha \bar{\beta}}=\langle \nabla_{\bar{\zeta}_\beta} \zeta_\alpha, \overline{\zeta_n} \rangle$, and write
\[
u_{\alpha \bar{\beta}}= \nabla_{\bar{\zeta}_\beta} (\nabla_{\zeta_\alpha}u) - \nabla_{\nabla_{\bar{\zeta}_\beta} \zeta_\alpha} u,
\]
for the complex Hessian. On the boundary $\p \Omega$, we will use $\lambda'$ to denote the eigenvalues of $u_{\alpha \bar{\beta}}$. With the above notation, 
\begin{equation} \label{eq:cplx_bdd_identity}
u_{\alpha \bar{\beta}}-\underline{u}_{\alpha \bar{\beta}}=-{1 \over 2} (u-\underline{u})_{x_n} \sigma_{\alpha \bar{\beta}},
\end{equation}
at $p \in \del \Omega$. 
\par
At a point $p \in \p \Omega$ with the above coordinates, instead of estimating $u_{x_n x_n}(p)$ directly, we will estimate $u_{n \bar{n}}(p)$, which is equivalent since $u_{n\bar{n}} = \frac{1}{4}\left(u_{x_{n}x_{n}}+u_{y_{n}y_{n}}\right)$ and $u_{y_n y_n}(p)$ is bounded. In analogy with the real case, to estimate $u_{n \bar{n}} \leq C$ it suffices to prove the existence of constants $R_0,c_0 >0$ such that for all $R \geq R_0$,
\begin{equation} \label{eq:cplx_lower_bdd}
g(\lambda'(u_{\alpha \bar{\beta}}), R) > \psi(z)+c_0, \qquad \forall z \in \del\Omega,
\end{equation}
where $\psi(z)=-e^{-h(z)}$. Proceeding as in the real case, for $z \in \del \Omega$ we define operators 
\[
\begin{aligned}
\tilde{G}(u_{\alpha \bar{\beta}}(x))&= - \exp \bigg(-A \sum_{\alpha=1}^{n-1} \arctan \lambda'_\alpha - A {\pi \over 2} \bigg), \\
\tilde{G}_0^{\alpha \bar{\beta}}&=\frac{\partial\tilde{G}}{\partial u_{\alpha \bar{\beta}}}(u_{\alpha \bar{\beta}}(0)).
\end{aligned}
\]
Consider the function $\tilde{G}(u_{\alpha \bar{\beta}})(z)-\psi(z)$ on $\del\Omega$.  Assume the minimum of this function is achieved at $p \in \del \Omega$. Choose coordinates as above so that $p$ is the origin and $\zeta_i(p) = \frac{\p}{\p z_i}$. As in the real case, it suffices to obtain an upper bound $u_{n \bar{n}}(0) \leq C$ to prove ~\eqref{eq:cplx_lower_bdd}.
\par
By Lemma ~\ref{G_0_tau_lemma}, for $u_{n \bar{n}}(0)$ large enough, there exists $\tau>0$ such that
\[
\tilde{G}_0^{\alpha\overline{\beta}}(\underline{u}_{\alpha\overline{\beta}}-u_{\alpha\overline{\beta}})(0)\geq\tau>0.
\] 
Let $\eta := \tilde{G}_0^{\alpha \bar{\beta}}\sigma_{\alpha \bar{\beta}}$. Arguing in the same way as the paragraph following Lemma~\ref{G_0_tau_lemma}, there exists a universal constant $\delta'>0$ such that $\eta \geq \epsilon$ in a small neighbourhood $B_{\delta'}(0)\cap \Omega$. We construct 
\[
\Phi =-{1 \over 2}(u-\underline{u})_{x_n}+\frac{1}{\eta}\tilde{G}_0^{\alpha \bar{\beta}}(\underline{u}_{\alpha \bar{\beta}}(z)-u_{\alpha \bar{\beta}}(0))-\frac{\psi(z)-\psi(0)}{\eta}.
\]
As in the real case, on the boundary $\partial\Omega$ we have $\Phi \geq 0$, and on $\del B_{\delta'}(0) \cap \Omega$ we have $\Phi \geq -C$.  Computing as in~\eqref{eq:LPhi_comp}, one estimates $L \Phi \leq C$ for a universal constant $C$. As before, we can find $A \gg B\gg 1$ large and universal to obtain
\[
\begin{aligned}
L(Av+B|z|^2+\Phi)&\leq 0\ \mathrm {in} \ B_{\delta'}(0)\cap \Omega \nonumber\\
 Av+B|z|^2+\Phi &\geq 0 \ \mathrm {on} \ \partial (B_{\delta'}(0)\cap \Omega).
\end{aligned}
\]
By the maximum principle, $\Phi_{x_n}(0)\geq-(Av+B|z|^2)_{x_n} \geq -C$, which gives $u_{x_n x_n}(0) \leq C$.  We conclude an upper bound $u_{n\bar{n}}(0) \leq C$. This completes the boundary $C^2$ estimate.

\subsection{Higher Order Estimates} We have therefore shown
\[
\| u \|_{L^\infty(\overline{\Omega})}+ \| Du \|_{L^\infty(\overline{\Omega})}+ \| D^2 u \|_{L^\infty(\overline{\Omega})} \leq C(\Omega,\|\underline{u}\|_{C^4(\overline{\Omega})}, \| h \|_{C^2(\overline{\Omega})},\delta).
\]
The $C^{2,\alpha}$ interior estimates follow from the Evans-Krylov theorem and an extension trick exploited introduced by Wang \cite{YW} in the study of the complex Monge-Amp\`ere equation.  The extension argument is used to extend the concave operator $G$ from the Hermitian matrices to the symmetric matrices.  We refer the reader to \cite{YW, TWWY} and, for instance, the proof of \cite[Lemma 6.1]{CJY}, but leave the details to the interested reader.  The boundary $C^{2,\alpha}$ estimates follow from Krylov \cite{K}. This establishes the a priori estimate
\[
\| u \|_{C^{2,\alpha}(\overline{\Omega})} \leq C(\Omega,\|\underline{u}\|_{C^4(\overline{\Omega})}, \| h \|_{C^2(\overline{\Omega})},\delta),
\]
and a continuity method argument as in \S \ref{cont_meth_sect} completes the proof of Theorem \ref{thm:CDirProb}.

\end{document}